\documentclass[12pt,a4paper]{article}
\usepackage[utf8]{inputenc}
\usepackage[english]{babel}
\usepackage{amsmath,amsfonts,amssymb,amsthm}
\title{Fractional weight multiplier systems on $\SUd$}
\author{Richard M. Hill\\ \small{University College London}\\
\small{r.m.hill@ucl.ac.uk}}
\date{August 2021}

\usepackage[english]{babel}
\usepackage[utf8]{inputenc}
\usepackage{amsmath,amssymb,amsfonts,amsthm}
\usepackage[dvipsnames]{xcolor}
\usepackage[colorlinks,linkcolor=Purple]{hyperref}

\usepackage[top=2.5cm, left=2cm, right=2cm, bottom=4.0cm]{geometry}

\newtheorem{definition}{Definition}
\newtheorem{corollary}{Corollary}

\newtheorem{lemma}{Lemma}
\newtheorem{proposition}{Proposition}
\newtheorem{theorem}{Theorem}

\newtheorem*{remark}{Remark}

\newcommand{\optional}[1]{}

\DeclareMathOperator{\ab}{ab}
\DeclareMathOperator{\cts}{{cts}}

\DeclareMathOperator{\Hom}{Hom}

\DeclareMathOperator{\meas}{meas}
\DeclareMathOperator{\nc}{n.c.}
\DeclareMathOperator{\Norm}{N}
\DeclareMathOperator{\stable}{stable}

\DeclareMathOperator{\PU}{PU}

\DeclareMathOperator{\SL}{SL}
\DeclareMathOperator{\SO}{SO}
\DeclareMathOperator{\Sp}{Sp}
\DeclareMathOperator{\Span}{Span}

\DeclareMathOperator{\SU}{SU}

\DeclareMathOperator{\Tr}{Tr}

\newcommand{\gog}{\mathfrak{g}}

\newcommand{\su}{\mathfrak{su}}

\newcommand{\A}{\mathbb{A}}
\newcommand{\C}{\mathbb{C}}

\newcommand{\G}{\mathbb{G}}

\newcommand{\Proj}{\mathbb{P}}
\newcommand{\Q}{\mathbb{Q}}
\newcommand{\R}{\mathbb{R}}
\newcommand{\Z}{\mathbb{Z}}

\newcommand{\cH}{\mathcal{H}}
\newcommand{\cO}{\mathcal{O}}

\newcommand{\ub}{{\underline{b}}}

\newcommand{\SUd}{\SU(d,1)}

\begin{document}
\maketitle

\begin{abstract}
	For a class of arithmetic subgroups $\Gamma \subset \SU(d,1)$
	we prove that for every positive integer $n$ there exists
	a subgroup $\Gamma_n$ of finite index in $\Gamma$, which
	lifts to the $n$-fold connected cover of of $\SUd$.
	Consequently $\Gamma_n$ has a multiplier system
	of weight $\frac{1}{n}$.
\end{abstract}

\tableofcontents
\newpage

\section{Introduction}

\subsection{Some History}
Let $G$ be a connected, non-compact, semi-simple real Lie group,
and $\Gamma\subset G$ an arithmetic subgroup.
By an automorphic form on $G$ with character $\chi :\Gamma \to\C$, one generally means a function $f: G \to \C$ satisfying (amongst several analytic properties) the transformation formula:
\[
	f(\gamma g) = \chi(\gamma)\cdot  f(g),\qquad
	\gamma\in\Gamma, \; g\in G.
\]
Suppose also that the fundamental group of $G$ is isomorphic to $\Z$, so that for each positive integer $n$ there is a connected $n$-fold cover:
\[
	1 \to \Z/n \to \tilde G^{(n)} \to G \to 1.
\]
Let $\tilde \Gamma^{(n)}$ denote the pre-image of $\Gamma$ in $\tilde G^{(n)}$.
By a \emph{fractional weight form with denominator $n$},
one means a function $f: \tilde G^{(n)} \to \C$, satisfying the transformation formula
\[
	f(\gamma g) = \chi(\gamma)\cdot f(g),\qquad
	\gamma\in\tilde\Gamma^{(n)}, \; g\in \tilde G^{(n)},
\]
where the character $\chi:\tilde\Gamma^{(n)} \to \C^\times$ is injective on the subgroup $\Z/n$.
The question of whether such forms exist is essentially
equivalent to the question of whether the group
$\tilde\Gamma^{(n)}$ splits as a direct sum of $\Z/n$ and a lift of $\Gamma$ to the $n$-fold cover of $G$.
We shall answer this question in some new cases in this paper.

\begin{quote}
	\textbf{Main question.}
	Given $\Gamma$ and $G$ as above, and a positive integer $n$, does there exist a subgroup $\Gamma_n \subset\Gamma$ of finite index, such that $\Gamma_n$ lifts to the $n$-fold cover $\tilde G^{(n)}$?
\end{quote}

\paragraph{The case of $\SL_2(\R)$.}
The question above was considered first in the case $G=\SL_2(\R)$ by Petersson \cite{Petersson III}, \cite{Petersson IV}. He showed that in this case, the answer is ``yes''.
More precisely, recall that there are two kinds of arithmetic subgroup $\Gamma$.
If $\Gamma$ has cusps, then some finite index subgroup $\Gamma'$ is free.
Hence $\Gamma'$ lifts to the $n$-fold cover for every $n$.
If on the other hand $\Gamma$ is cocompact, then some finite index subgroup $\Gamma'\subset\Gamma$ is a surface group.
In this case, it is sufficient to take $\Gamma_n$ to be any subgroup of index $n$ in $\Gamma'$.

The fractional weight forms found by Petersson have recently been studied by Ibukiyama and others
(see for example \cite{ibukiyama}).

\paragraph{Metaplectic forms.}
Suppose that $G$ is the group $\G(\R)$ of real points of some algebraic group $\G$ over $\Q$,
and that $\Gamma$ is a congruence subgroup of $\G(\Q)$.
It has been known for some time that any such $\Gamma$ contains a congruence subgroup $\Gamma_2$ which has half-integral weight automorphic forms. Such forms are called \emph{metaplectic forms}, and arise because of the splitting of the $2$-fold metaplectic cover
of $\G(\A)$ over the subgroup $\G(\Q)$ (here we are writing $\A$ for the ad\`ele ring of $\Q$).
This result was proved in its most general form by Deligne \cite{Deligne2}.
Such metaplectic forms have been hugely studied.
The basic examples are theta series.

There are also analogous ``metaplectic forms'' of other
fractional weights (i.e. with $n>2$).
However, these forms do not arise
from a connected cover of a real Lie group, but rather from an $n$-fold cover of the group $\G(\A_f)$ of finite-ad\`ele points of $\G$.
These metaplectic forms with $n>2$ are not the subject of this paper.

\paragraph{The congruence subgroup problem.}
As far as the author knows, the next person to consider the question above was Deligne \cite{Deligne}, who considered the case
$G=\Sp_{2d}(\R)$.
Deligne proved that no arithmetic subgroup of $\Sp_{2d}(\Z)$ for $d \ge 2$ can lift to the $n$-fold cover of $\Sp_{2d}(\R)$ for $n>2$.
Deligne's proof uses the congruence subgroup property,
which had been established for such groups
by Bass, Milnor and Serre \cite{BassMilnorSerre}.
Indeed, Deligne's result applies to any case where the congruence subgroup property is known.
Thus we have
\begin{quote}
	If $\Gamma$ has the congruence subgroup property and
	$n>2$, then the answer to the main question above is ``no''.
\end{quote}
Recall that the congruence subgroup property is conjectured to hold for groups of real rank at least $2$.
Hence the question above remains most interesting for groups $G$ of real rank $1$.
Up to an isomotry, the only such group with fundamental group $\Z$ is $\SUd$. We focus on this case is this paper.

Recently Stover and Toledo \cite{stover toledo},
have proved for several arithmetic subgroups of $\SU(1,2)$
that there do indeed exist subgroups $\Gamma_n$ which lift to the $n$-fold cover.
The main result of this paper generalizes theirs.

\subsection{Statement of results}

There are two kinds of arithmetic subgroup of $\SUd$, which we shall call (following \cite{BlasiusRogawski}) the ``first kind'' and the ``second kind''.
This paper is concerned with arithmetic groups of the first kind.
We briefly review the construction.

Let $F$ be a totally real number field and $E/F$ a totally
complex quadratic extension.
We shall write $z \mapsto\bar z$ for the
non-trivial Galois automorphism of $E$ which fixes $F$.
A matrix $J \in M_{d+1}(E)$ is called Hermitian if $\bar J^t=J$.
We shall fix a Hermitian matrix $J \in M_{d+1}(E)$,
which has signature $(d,1)$ at one of the complex field embeddings of $E$, and signature either $(0,d+1)$ or $(d+1,0)$ at all other complex field embeddings of $F$.
Such a choice of $J$ allows us to
 define a group scheme $\G$ over $\Z$ by
\begin{equation}
	\label{eqn:G defn}
	\G(A) = \{ g \in \SL_{d+1}(\cO_E \otimes_{\Z} A ):
		\bar g^t J g = J
	\},
	\qquad
	\text{$A$ a $\Z$-algebra}.
\end{equation}
We have $\G(\R)= \SUd \times \SU(d+1)^{[F:\Q]-1}$.
The projection of $\G(\Z)$ in $\SUd$ is an arithmetic subgroup.
More generally, we shall call any subgroup of $\SUd$ commensurable with $\G(\Z)$ an arithmetic group of the \emph{first kind}.
Such groups are cocompact if the Hermitian form defined by $J$ is non-isotropic on $E^{d+1}$.
This is always the case if $F \ne \Q$.

\begin{theorem}
	\label{thm:main}
	Let $\Gamma$ be a cocompact arithmetic subgroup of $\SUd$
	of the first kind.
	Then for every positive integer $n$ there exists
	a subgroup $\Gamma_n$ of finite index in $\Gamma$
	which lifts to the connected $n$-fold cover of $\SUd$.
\end{theorem}

We recall in \autoref{section:background} the notion
of a fractional weight multiplier system.
With this notation, \autoref{thm:main} is equivalent to
the following more classical looking statement:

\begin{theorem}
	\label{thm:main multiplier}
	Let $\Gamma$ be a cocompact arithmetic subgroup of $\SUd$
	of the first kind.
	For every natural number $n>0$ there exists
	a subgroup $\Gamma_n$ of finite index in $\Gamma$
	and a weight $\frac 1n$ multiplier system on $\Gamma_n$.
\end{theorem}

In \autoref{section:example} we give an example of such a sequence of subgroups $\Gamma_n$ together with the corresponding multiplier systems (\autoref{thm:example} and \autoref{eqn:example multiplier system}).
Actually, we expect \autoref{thm:main} to hold also in the case that $\Gamma$ has cusps, and the group $\Gamma$ in \autoref{thm:example} does indeed have cusps.

There is an equivalent purely group-theoretical restatement
of \autoref{thm:main}:

\begin{theorem}
	\label{thm:main2}
	Let $\Gamma$ be a cocompact arithmetic subgroup of $\SUd$
	of the first kind and let $\tilde\Gamma$ be the preimage of $\Gamma$ in the universal cover of $\SUd$.
	Then $\tilde\Gamma$ is residually finite
\end{theorem}

There are also arithmetic subgroups of $\SUd$ of the ``second kind'' (see \cite{BlasiusRogawski}), whose construction involves a non-commutative division algebra.
It is not known whether arithmetic groups of the second kind
have the congruence subgroup property or not.
If they do have the congruence subgroup property, then
the result of Deligne \cite{Deligne} implies that there are no fractional weight
multiplier systems apart from those of half-integral weight.
On the other hand, if there does exist a non-congruence subgroup
$\Gamma$ of the second kind, and if $H^1(\Gamma,\C)\ne 0$, then there would be no difficulty in extending \autoref{thm:main}
to this case.

\medskip

The theorems are easily deduced from the following
result.

\begin{lemma}
	\label{key lemma}
	Let $\Gamma$ be a cocompact arithmetic subgroup of $\SUd$
	of the first kind.
	Let $\sigma\in H^2(\Gamma,\C)$ be the image of the invariant K\"ahler form on the symmetric space for $\SUd$.
	Then there exists a congruence subgroup
	$\Gamma'\subset\Gamma$ and elements $\phi_1 ,\ldots, \phi_r, \psi_1,\ldots,\psi_r\in H^1(\Gamma',\C)$ such that
	the following holds in $H^2(\Gamma',\C)$:
	\[
		\sigma = \phi_1 \cup \psi_1 + \cdots + \phi_r \cup \psi_r.
	\]
\end{lemma}

The theorems are proved in \autoref{section:thm proof}, assuming \autoref{key lemma}, which is proved in \autoref{section:lemma proof}.
In the final section of the paper, we give an alternative and more constructive proof of \autoref{key lemma} in the special case that
$H^{2}(\Gamma,\C)$ is $1$-dimensional.
This proof leads to the following bound on the size of $H^1$ in this case:

\begin{theorem}
	Let $\Gamma$ be a cocompact arithmetic subgroup of
	$\SUd$ of the first kind, such that $H^{2}(\Gamma,\C)$ is $1$-dimensional.
	If $\Gamma'$ is any neat normal subgroup of $\Gamma$ of finite index with $H^1(\Gamma',\C) \ne 0$, then we must have
	\[
		\dim H^1(\Gamma',\C) \ge 2d.
	\]
\end{theorem}

This bound is proved by analysing the rate of growth
of the index $[\Gamma:\Gamma_n]$ for the subgroups
constructed in \autoref{thm:main} as $n$ tends to infinity.

\paragraph{Remarks}
\begin{enumerate}

\item
During the preparation of this paper, the author became aware of  the work of Stover and Toledo \cite{stover toledo},
in which several special cases of \autoref{thm:main} are proved.
In particular, they treat the case described in detail in
\autoref{section:example}.

\item
In the case $d=1$, the group $\SU(1,1)$ is isomorphic to $\SL_2(\R)$, and the theorems were proved by Petersson
\cite{Petersson III}, \cite{Petersson IV}, as described above.
We shall always assume from now on that $d \ge 2$.

\item
The results are also known in the case $n=2$.
An explicit multiplier systems of weight $\frac{1}{2}$
on $\SU(2,1)$ has been found by Jalal \cite{Lina Jalal}
using the metaplectic splitting.

\item
The existence of a weight $\frac 1n$ multiplier
system on $\Gamma_n$ implies the existence of
automorphic forms of weight $\frac kn$ and level $\Gamma_n$ for all sufficiently large integers $k$. For example, such forms may be constructed as
Eisenstein or Poincar\'e series.
We make no attempt to describe such forms here.

Some holomorphic forms of weight $\frac 12$ on $\SUd$ have already been found by Wang and Williams \cite{WangWilliams} using Borcherds products. In fact they construct forms on $\SO(2d,2)$. However, the restriction of a
half-integral weight form from $\SO(2d,2)$ to $\SUd$ must also have half-integral weight.

\item
The subgroups $\Gamma_n$ constructed in \autoref{thm:main}
are non-congruence subgroups for $n$ large enough, and are constructed using homomorphisms $\Gamma'\to \Z$ for certain congruence subgroups $\Gamma'$ of $\Gamma$.

Based on very limited computational evidence (using the code at  \cite{hill code}), the author does not believe that \autoref{thm:main} would hold for any congruence subgroups $\Gamma_n$ apart from in the cases where the number field $E$ contains a primitive $n$-th root of unity.

For arithmetic groups of the second kind in $\SUd$,
one may prove that there is no congruence subgroup $\Gamma_n$ satisfying the conditions of \autoref{thm:main} for $n >2$.
The proof is essentially the same as Deligne's proof in \cite{Deligne}, but with the congruence subgroup property replaced by the ``abelian congruence subgroup property'' of P. Boyer \cite{Boyer}.
In this context, the abelian congruence subgroup property states that if $\Gamma$ is a congruence subgroup and $\phi : \Gamma \to \Z/n$ is a homomorphism then $\ker\phi$ is also a congruence subgroup.

\end{enumerate}

\emph{Acknowledgement.}
The author would like to thank Lars Louder, Frank Johnson, Yiannis Petridis and Eberhard Freitag for many useful discussions.

\section{Background on $\SUd$}
\label{section:background}
We shall realise $\SUd$ as the group of matrices $g$ in $\SL_{d+1}(\C)$ satisfying $\bar g^t J g=J$, where $J$ is the following matrix:
\[
	J = \begin{pmatrix}
		0 & 0 & 1 \\
		0 & I_{n-1} & 0 \\
		1 & 0 & 0
	\end{pmatrix}.
\]
(Here and later we write $I_n$ for the $n\times n$ identity matrix).
The group $\SUd$ is the group of isometries of the following Hermitian form on $\C^{d+1}$ given by $\langle v, w \rangle = \bar v^t J w$.
Let
\[
	\cH = \left\{ \tau \in \C^d :
	\left\langle \begin{pmatrix} \tau \\ 1\end{pmatrix},\begin{pmatrix} \tau \\ 1\end{pmatrix}\right\rangle < 0 \right\}.
\]
The set $\cH$ is a homogeneous space for the group $\SUd$.
Suppose $g\in\SUd$ is described as a block matrix
$g=\begin{pmatrix} A&B\\C&D \end{pmatrix}$,
where $A$ is a $d\times d$ matrix, $D$ is a complex number, etc.
Then the action of $\SUd$ on $\cH$ is given by
\[
	g*\tau
	=(C\tau+D)^{-1} \cdot (A\tau+B).
\]
We also define
\begin{equation}
	\label{eqn: j definition}
	j(g,\tau)
	= C \tau + D.
\end{equation}
The function $j(g,\tau)$ satisfies the usual condition of a multiplier system:
\begin{equation}
	\label{eqn: j multiplier system}
	j(gh,\tau) = j(g,h\tau)\cdot j(h,\tau),
	\qquad
	g,h \in \SUd.
\end{equation}
One may use this multiplier system to define modular forms on
$\SU(d,1)$, and such objects have been extensively studied
(see for example \cite{holzapfel}).

\begin{definition}
	\label{defn:frac wt multiplier system}
	Let $\Gamma$ be an arithmetic subgroup of $G$.
	A function $\ell : \Gamma \times \cH \to \C^\times$ is called
	a \emph{fractional weight multiplier system} of weight $\frac{n}{m}$
	if it satisfies the following conditions:
	\begin{enumerate}
		\item
		$\ell$ is a multiplier system, i.e.
		$\ell(gh,\tau)= \ell(g,h\tau)\cdot \ell(h,\tau)$
		for all $g,h\in \Gamma$ and $\tau \in \cH$;
		\item
		there exists a character
		$\chi : \Gamma \to \C^\times$ such that
		\[
			\ell(g,\tau)^m = \chi(g)\cdot j(g,\tau)^n
		\]
		for all $g\in \Gamma$ and $\tau \in \cH$;
		\item
		For each $g\in \Gamma$ the function $\tau\mapsto \ell(g,\tau)$ is continuous (and hence holomorphic) on $\cH$.
	\end{enumerate}
\end{definition}

\subsection{The universal cover and fractional weight multiplier systems}

The Lie group $\SUd$ has fundamental group $\Z$.
We shall write $\widetilde{\SUd}$ for its universal cover.
We therefore have a central extension of groups
\[
	1 \to \Z \to \widetilde{\SUd} \to \SUd \to 1.
\]
For the purposes of the calculations in \autoref{section:example},
it will be helpful to have a specific 2-cocycle $\sigma$
representing this group extension.
We describe such a cocycle now.

Define a function $X : \SUd\to \C$ by
\[
	X(g) = \begin{cases}
		-g_{d+1,1} & \text{if $g_{d+1,1}\ne 0$,}\\ 
		g_{d+1,d+1} & \text{if $g_{d+1,1} = 0$.}
	\end{cases}
\]
Here we are writing $g_{i,j}$ for the complex number in the $i$-th row and the $j$-th column of the matrix $g$.

\begin{lemma}
	\label{lemma:half-plane}
	Let $g\in \SUd$ and $\tau \in \cH$ we have $X(g) \ne 0$ and
	the complex number $\frac{j(g,\tau)}{X(g)}$ has positive real part, where $j(g,\tau)$ is defined in \autoref{eqn: j definition}.
\end{lemma}

\begin{proof}
	Let $\begin{pmatrix}a&\ub&c\end{pmatrix}$ be the bottom row of the matrix $g$, where $a$ and $c$ are complex numbers and $\ub\in \C^{d-1}$. In order that $g\in \SUd$ we must have $	a\bar c + ||b||^2 + c\bar a =0$, where $||\ub||$ denotes the standard (positive definite) norm on $\C^{d-1}$.
	If $a=0$ then we must have $\ub=0$. In this case $X(g)=j(g,\tau)=c$ for all $\tau\in \cH$ so the result is true in this case.
	
	Suppose now that $a\ne 0$.
	In this case we have $X(g)=-a$, so in particular $X(g) \ne 0$.
	We may choose a diagonal matrix $t\in \SUd$, such that $t_{d+1,d+1}=a^{-1}$.
	It is trivial to check that $X(tg) = -1$ and $j(tg,\tau)=a^{-1}j(g,\tau)$.
	In view of this, it's sufficient to prove the lemma in the case $a=1$, so that we have $2\Re(c) = - ||\ub||^2$.
	We must check in this case that $j(g,\tau)$ has negative real part.
	
	Let $\tau = \begin{pmatrix} \tau_1 \\ \tau_2\end{pmatrix}$
	with $\tau_1 \in \C$ and $\tau_2\in \C^{d-1}$.
	Since $\tau\in \cH$ we must have
	$2\Re(\tau_1) + ||\tau_2||^2 < 0$.
	We can finish the proof using the Cauchy--Schwarz inequality :
	\[
		\Re(j(g,\tau))
		= \Re( \tau_1  + b \tau_2 + c) 
		< -\frac{||\tau_2||^2}{2} + \Re (\ub \tau_2)  - \frac{||\ub||^2}{2} \le 0.
	\]
\end{proof}

The lemma allows us to define, for each $g\in \SUd$ a branch $\tilde j(g,-)$ of the logarithm of $j(g,-)$ as follows:
\begin{equation}
	\label{eqn:j tilde}
	\tilde j(g,\tau) = \log\left( \frac{j(g,\tau)}{X(g)} \right) + \log(X(g)),
\end{equation}
where each logarithm in the right hand side is defined to be continuous away from the negative real axis and satisfy $-	\pi < \Im ( \log z) \le \pi$.
Note that by \autoref{lemma:half-plane}, for each fixed $g\in\SUd$ the function $\tilde j(g,\tau)$ is continuous in $\tau$.

Our multiplier system condition (\autoref{eqn: j multiplier system}) on $j(g,\tau)$ implies the following formula for $\tilde j$:
\[
	\tilde j(gh,\tau) \equiv \tilde j(g,h\tau) + \tilde j(h,\tau) \mod 2\pi i \Z.
\]
In particular, we may define an integer $\sigma(g,h)$ by
\begin{equation}
	\label{eqn:sigma definition}
	\sigma(g,h) = \frac{1}{2\pi i} \Big(\tilde j(gh,\tau) - \tilde j(g,h\tau) - \tilde j(h,\tau)\Big).
\end{equation}
The right hand side of \autoref{eqn:sigma definition} is independent of $\tau$, since it is a continuous $\Z$-valued function of $\tau$.

\begin{proposition}
	The function $\sigma$ defined in \autoref{eqn:sigma definition} is a measurable 2-cocycle, whose cohomology class in $H^2_{meas}(\SUd,\Z)$
	corresponds to the universal cover of $\SUd$.
\end{proposition}

\begin{proof}
	It is an easy exercise to verify the 2-cocycle relation, and the function $\sigma$ is evidently measurable.
	Hence by Calvin Moore's theory of measurable cohomology
	\cite{Moore},
	there is a central extension of Lie groups
	corresponding to $\sigma$.
	It remains to show that this extension is the universal cover.
	For the moment, we'll write $\widetilde{\SUd}$ for the
	extension of $\SUd$ corresponding to $\sigma$.
	Explicitely, $\widetilde{\SUd}$ is the set $\SUd \times \Z$, with
	the group operation given by
	\[
		(g,n)\cdot (g',n')
		=(gg', n+n'+\sigma(g,g')).
	\]
	To prove that $\widetilde{\SUd}$ is the universal cover of $\SUd$, it's sufficient to prove that $\widetilde{\SUd}$ is connected.
	(Note that the topology on $\widetilde{\SUd}$ is not the
	product topology on $\SUd \times \Z$).
	
	Let $\tilde T$ be the pre-image in $\widetilde{\SUd}$ of the following torus in $\SUd$:
	\[
		T = \left\{t_z:
		z\in \C, |z|=1\right\},
		\qquad
		t_z=\begin{pmatrix} z \\& \bar z^2\\ && I_{d-2}\\ &&& z \end{pmatrix}.
	\]
	Note that on $T$, we have simply $\tilde j(t_z,\tau) = \log(z)$.
	An elementary calculation shows that there is an
	isomorphism $\Psi:\tilde T \to \R$ given by
	\[
		\Psi(t_z,n) = \frac{1}{2\pi i} \log(z) - n.
	\]
	Check:
	\begin{align*}
		\Psi((t_z,n)(t_{z'},n'))
		&= \Psi((t_{zz'} , n + n' + \sigma(t_z,t_{z'}))\\
		&= \frac{1}{2\pi i}\log(zz') -  \left(n + n' + \frac{1}{2\pi i}\big(\log(zz')-\log(z)-\log(z')\big) \right)\\
		&= \frac{1}{2\pi i}\log(z) -  n  + \frac{1}{2\pi i}\log(z')-n'.
	\end{align*}
	The map $\Psi$ and its inverse are measurable.
	Since measurable homomorphisms of Lie groups are continuous, it follows that $\Psi$ is a homeomorphism.
	In particular $\tilde T$ is connected.
	Hence all elements of the subgroup $\Z \subset \widetilde{\SUd}$ are connected by paths in $\tilde T$ to the identity element.
	This implies that $\widetilde{\SUd}$ is connected.
\end{proof}

We'll now describe the cohomological meaning of the
fractional weight multiplier systems defined above.
Recall that we have a cohomology class $\sigma \in H^2(\SUd,\Z)$.
For any integer $n>0$, we shall write $\sigma_n$ for the
image of $\sigma$ in $H^2(\SUd,\Q/n\Z)$.

\begin{proposition}
	\label{prop:multiplier system cohomology}
	There exists a weight $\frac{1}{n}$ multiplier system on
	$\Gamma$ if and only if the restriction of $\sigma_n$ to
	$\Gamma$ splits as an element of $H^2(\Gamma,\Q/n\Z)$.
\end{proposition}

\begin{proof}
	Suppose $\sigma_n$ splits on $\Gamma$.
	We have a long exact sequence
	\[	
		\to H^2(\Gamma,n\Z)
		\to H^2(\Gamma,\Q)
		\to H^2(\Gamma,\Q/n\Z)
		\to
	\]
	It follows that $\sigma$ is in the image of $H^2(\Gamma,n\Z)$.
	Hence there is another cocycle $\sigma' \in Z^2(\Gamma,n\Z)$ and a map $\kappa : \Gamma \to\Q$ such that
	\[
		\sigma(g,h) = \sigma'(g,h) + \kappa(gh)-\kappa(g)-\kappa(h)
		\qquad
		g,h \in \Gamma.
	\]
	This implies
	\[
		\exp\left(\frac{2\pi i \cdot \sigma(g,h)}{n}\right)
		=
		\exp\left( 	\frac{2	\pi i}{n} (\kappa(gh)-\kappa(g)-\kappa(h)) \right). 
	\]
	Equivalently, the function
	\[
		\ell(g,\tau) = \exp\left( \frac{\tilde j(g,\tau) - 2 \pi i \cdot\kappa(g)}{n} \right),
	\]
	is a weight $\frac{1}{n}$ multiplier system on $\Gamma$,
	where $\tilde j(g,\tau)$
	is defined in \autoref{eqn:j tilde}.
	This argument may be reversed to prove the converse.
\end{proof}

\subsection{Stable cohomology}
Let $\Gamma$ be a cocompact arithmetic subgroup of $\SUd$ of the first kind. Thus $\Gamma$ is commensurable with some $\G(\Z)$
with $\G$ as defined in \autoref{eqn:G defn}.

We shall define the
\emph{stable cohomology} $H^\bullet_{\stable}(\C)$ by
\[
	H^\bullet_{\stable}(\C)
	=
	\lim_{\to} H^\bullet(\Gamma',\C),
\]
where the limit is taken over all congruence subgroups $\Gamma'$ of
$\Gamma$.
The restriction maps $H^\bullet(\Gamma',\C) \to H^\bullet(\Gamma'',\C)$ are all injective, so the direct limit may be thought of an a union.
We have a cocycle $\sigma\in H^2(\SUd,\Z)$. We shall often abuse notation and write $\sigma$ for the image of this cocycle in
$H^2(\Gamma,\Z)$, $H^2(\Gamma,\Q)$, $H^2(\Gamma,\C)$ or $H^2_{\stable}(\C)$. The meaning should be clear from the context.

Recall that the restriction maps $H^\bullet(\Gamma,\C) \to H^\bullet(\Gamma',\C)$ are compatible with cup product.
This implies that we have a cup product operation on the stable
cohomology, which coincides with the usual cup product on each $H^\bullet(\Gamma,\C)$.

We recall now some standard facts about stable cohomology, which may be found (in the current context) in
\cite{BlasiusRogawski} and more generally in \cite{BorelWallach}.
There is an action of $\G(\Q)$ on the stable cohomology,
which extends to a smooth action of the totally disconnected group
$\G(\A_f)$, where $\A_f$ denotes the ring of finite ad\`eles of $\Q$.
By the Strong Approximation Theorem, there is a bijective correspondence between the congruence subgroups $\Gamma(K_f)$ of $\G(\Q)$ and the compact open subgroups $K_f$ of $\G(\A_f)$, where $K_f$ is the closure of $\Gamma(K_f)$ in $\widetilde{G(\A_f)}$, or equivalently $\Gamma(K_f)$ is the intersection of $K_f$ with $\G(\Q)$.
We may recover the cohomology of any congruence subgroup from
the stable cohomology by
\[
	H^\bullet(\Gamma(K_f),\C) = \left(H^\bullet_{\stable}(\C)\right)^{K_f}.
\]

The smooth representation $H^\bullet_{\stable}(\C)$ of $\G(\A_f)$
is semi-simple, and has a decomposition into irreducibles of the following form:
\begin{equation}
	\label{eqn:matsushima}
	H^\bullet_{\stable}(\C)
	=\bigoplus_{\pi}
	H^\bullet(\gog,K,\pi_\infty) \otimes \pi_f,
\end{equation}
where $\pi$ runs through certain automorphic representations of $\G(\A)$ which arise in $L^2(\G(\Q)\backslash \G(\A))$.
In the relative Lie algebra cohomology on the right hand side of \autoref{eqn:matsushima}
we have $\gog=\su(d,1)$ and $K$ is a maximal compact subgroup of $\SUd$.
In this notation, we have $\pi = \pi_\infty \otimes \pi_f$, where $\pi_\infty$ is an irreducible representation
of $\SU(d,1)$ and $\pi_f$ is an irreducible representation
of $\G(\A_f)$.

One of the automorphic representations arising in the decomposition
(\autoref{eqn:matsushima})
is the one dimensional subspace $\pi=\C$ of constant functions.
For this trivial representation we have $\pi_\infty=\C$ and $\pi_f=\C$.
Thus the stable cohomology contains $H^\bullet(\gog,K,\C)$ as a direct summand with the trivial action of $\G(\A_f)$.
For any non-trivial automorphic representation $\pi$ occurring in the decomposition, $\pi_f$ is an infinite dimensional irreducible
representation of $\G(\A_f)$.
From this, it follows that
\[
	\left( H^\bullet_{\stable}(\C) \right)^{\G(\A_f)}
	= H^\bullet(\gog,K,\C).
\]

The cohomology groups
$H^\bullet_{\cts}(\gog,K,\C)$ are canonically isomorphic to the cohomology of the compact dual of the symmetric space $\cH$.
This compact dual is the complex projective space $\Proj^d(\C)$ and
has the following cohomology groups
\[
	H^r_{\cts}(\gog,K,\C)
	=
	H^r(\Proj^d(\C),\C)
	=\begin{cases}
		\C & r=0,2,4,\ldots,2d,\\
		0 & \text{otherwise}.
	\end{cases}
\]
A generator of $H^2_{\cts}(\gog,K,\C)$ is called an invariant K\"ahler class; such a generator is the de Rham cohomology class of
the $2$-form given by the imaginary part of the invariant
Hermitian metric on $\cH$. In fact, such a class is only determined up to constant multiple.
The other non-zero cohomology groups $H^{2r}(\gog,K,\C)$ for
$r=2,\ldots,d$ are generated by the cup-powers of the K\"ahler class.

The natural map
$H^\bullet_{\cts}(\gog,K,\C) \cong H^\bullet_{\cts}(\SUd,\C) \to H^\bullet_{\meas}(\SUd,\C)$ is an isomorphism.
We have described in \autoref{eqn:sigma definition} a cocycle $\sigma$, whose cohomology
class generates $H^2_{\meas}(\SUd,\Z)$.
Hence the image of $\sigma$ in $H^2(\gog,K,\C)$
is an invariant K\"ahler class.
In summary, we have
\[
	(H^r_{\stable}(\C))^{\G(\A_f)}
	=
	\begin{cases}
		\Span\{\sigma^{s}\} &
		\text{if $r=2s$ with $s = 0 ,\ldots, d$,}\\
		0 & \text{otherwise.}
	\end{cases}
\]

\section{Example : $\SU(2,1)$ over the Eisenstein integers}
\label{section:example}

This section concerns computer calculations.
The code for these calculations in runs on sage version 9.0, and is available at \cite{hill code}.

Let $\zeta=e^{2\pi i /3}$.
In this section we shall consider arithmetic subgroups of the following group:
\[
	\Gamma = \{ g \in \SL_3(\Z[\zeta]) : \bar g^t J g = J \},
	\qquad
	J = \begin{pmatrix}
		0&0&1 \\
		0&1&0 \\
		1&0&0
	\end{pmatrix}.
\]
The group $\Gamma$ is an arithmetic subgroup of
$\SU(2,1)$ of the first kind.
The Hermitian form on $\Q(\zeta)^3$ defined by the matrix $J$ is isotropic,
so $\Gamma$ has cusps.

Let $\Gamma(\sqrt{-3})$ denote the principal congruence subgroup in $\Gamma$ of level $\sqrt{-3}$.
The group $\Gamma(\sqrt{-3})$ decomposes as a direct sum:
\[
	\Gamma(\sqrt{-3}) = \Upsilon  \oplus \mu_3,
\]
where $\mu_3$ is the centre of $\SU(2,1)$ generated by $\zeta I_3$,
and $\Upsilon $ is a subgroup of index $3$ defined as follows:
\[
	\Upsilon 
	=
	\{ (g_{i,j}) \in \Gamma(\sqrt{-3}) : g_{1,1} \equiv 1 \bmod 3 \}.
\]

It will be useful to have the following notation for
upper-triangular elements of $\Upsilon$:
\[
	n(z,x)
	=
	\begin{pmatrix}
		1 & \sqrt{-3} \cdot z & \frac{-3\Norm(z) + x\sqrt{-3}}{2}\\[2mm]
		0 & 1 & \sqrt{-3}\cdot\bar z\\[2mm]
		0 & 0 & 1
	\end{pmatrix},
	\qquad
	z\in \Z[\zeta],\quad x\in\Z,\quad x\equiv \Norm(z) \bmod 2.
\]
Here and later we write $\Norm$ and $\Tr$ for the norm and trace maps from
$\Q(\zeta)$ to $\Q$.

\begin{lemma}
	\label{lemma:Euclidean algorithm}
	The group $\Upsilon $ is generated by the following elements:
	\[
		n(z,x), \qquad n(z,x)^t,
		\qquad
		\text{where $z\in \Z[\zeta]$ and $x\in\Z$ satisfy $x \equiv \Norm (z) \bmod 2$}.
	\]
\end{lemma}

\begin{proof}
	The proof mimics the use of the Euclidean algorithm to find generators of $\SL_2(\Z)$; it will be given in a forthcoming paper.
\end{proof}

For the calculations described below, it is convenient to have a fairly small finite presentation of $\Upsilon$.
To obtain such a presentation, one begins with the presentation of
\cite{FalbelParker} for the group $\PU(2,1)(\Z[\zeta])$,
and then applies the Reidemeister--Schreier algorithm to the subgroup $\Upsilon $.
The resulting presentation of $\Upsilon$ is rather large (52 generators and 223 relations), and \autoref{lemma:Euclidean algorithm} shows that we only need a small number of generators.
After rewriting the large presentation in terms of the small set
of generators, and simplifying the resulting presentation using GAP, we obtain a presentation for $\Upsilon$ with the following five generators
\[
	n_1=n(1,1),	\quad
	n_2=n(\zeta,1),\quad
	n_3=n(0,2),\quad
	n_4 = n_1^t,\quad
	n_5 = n_3^t.
\]
and with the following thirteen relations.
\[
	\begin{matrix}
		[n_1,n_3], \qquad
		[n_2, n_3], \qquad
		[n_4, n_5],\qquad
		( n_3  n_5 )^3,\\[2mm]
	    n_3 n_2 n_1 n_2^{-1} n_3 n_1^{-1} n_3,\qquad
    	(n_1^{-1} n_3 n_4^{-1})^3,\\[2mm]
    	n_5^{-1} n_2 n_5 n_4^{-1} n_1^{-1} n_2^{-1} n_3 n_4 n_3^{-1} n_1,\\[2mm]
	    n_4^{-1} n_1^{-1} n_3 n_5 n_2 n_1 n_5^{-1} n_4 n_2^{-1} n_3^{-1},\\[2mm]
	    n_5^{-1} n_4 n_1 n_5 n_3^{-1} n_1^{-1} n_2^{-1} n_4^{-1} n_3^{-1} n_3^{-1} n_2,\\[2mm]
	    n_5^{-1}n_2 n_1 n_5^{-1} n_4 n_1 n_4 n_1 n_5^{-1} n_4 n_2^{-1},\\[2mm]
	    n_3 n_5 n_1 n_4 n_5^{-1} n_2^{-1} n_4 n_3^{-1} n_1 n_4 n_2 n_1,\\[2mm]
		n_3^{-1} n_1 n_4 n_2 n_3 n_1 n_5^{-1} n_1^{-1} n_4^{-1} n_5 n_1^{-1} n_2^{-1},\\[2mm]
		n_4^{-1} n_3^{-1} n_5 n_3 n_1^{-1} n_4^{-1} n_2
		n_1 n_3^{-1} n_4 n_1 n_5^{-1} n_4 n_2^{-1} n_1^{-1} n_3.
    \end{matrix}
\]

Using this presentation, we can easily calculate the abelianization
$\Upsilon^{\ab}= \Upsilon/[\Upsilon,\Upsilon]$ of $\Upsilon$.
Each relation gives us a row vector in $\Z^5$, whose entry in column $r$
is the multiplicity of the generator $n_r$ in the relation.
These row vectors form an $13 \times 5$ matrix.
If we row reduce the matrix over $\Z$,
 then we obtain the following matrix:
\[
	\begin{pmatrix}
		3 & 0 & 0 & 3 & 0 \\
		0 & 0 & 3 & 0 & 0 \\
		0 & 0 & 0 & 0 & 3 \\
	\end{pmatrix}
\]
This gives us the following relations in the abelianization:
\[
	(n_1 n_4)^3=1, \qquad n_3^3=1, \qquad n_5^3=1.
\]
In particular $\Upsilon^{\ab}$ is isomorphic to $\Z^2 \oplus (\Z/3)^3$.

\begin{lemma}
	There exists a surjective homomorphism $\phi : \Upsilon  \to \Z[\zeta]$, defined on the generators of \autoref{lemma:Euclidean algorithm} by
	\[
		\phi(n( z,x)) = z,\qquad 
		\phi(n( z,x)^t) = -\bar z.
	\]
	The vector space $H^1(\Upsilon,\C)$ is 2-dimensional, and has
	$\{\phi,\bar \phi\}$ as a basis.
\end{lemma}

\begin{proof}
	One sees from our presentation of the abelianization, that there is a homomorphism $\phi$ defined on generators by
	\[
		\phi(n_1) = 1,\qquad
		\phi(n_2) = \zeta,\qquad
		\phi(n_3) = 0,\qquad
		\phi(n_4) = -1,\qquad
		\phi(n_5) = 0,\qquad
	\]
	From this, it is trivial to check that $\phi(n(z,x))=z$.
	To verify our formula for $\phi(n(z,x)^t)$,
	we need only check that $\phi(n(\zeta,1)^t)=-\bar\zeta$.
	We have $n(\zeta,1)^t = n_3^{-1} n_1 n_4 n_1 n_3^{-2} n_2$, and therefore
	\[
		\phi( n(\zeta,1) )
		=
		1+\zeta = - \bar\zeta.
	\]
	Our calculation of $\Upsilon^{\ab}$ implies that $H^1(\Upsilon,\C)$ is 2-dimensional.
	Clearly $\phi$ and $\bar\phi$ are linearly independent,
	so they form a basis for $H^1(\Upsilon,\C)$.
\end{proof}

The lemma allows us to define a tower of subgroups
$\Upsilon_{\nc}(I)$ indexed by
non-zero ideals $I\subset \Z[\zeta]$ as follows:
\[
	\Upsilon_{\nc}(I) = \phi^{-1} (I).
\]

\begin{theorem}
	\label{thm:example}
	For each non-zero ideal $I\subset \Z[\zeta]$
	there exists a multiplier system on
	$\Upsilon_{\nc}(2\cdot I)$ of weight $\frac{1}{\Norm(I)}$.
\end{theorem}

We give a formula for the multiplier system on in \autoref{eqn:example multiplier system} below.

To prove the theorem, we first replace $\sigma$ by a more convenient 2-cocycle on $\Upsilon$.
This is achieved by the following lemma.

\begin{lemma}
	\label{lemma:Sigma-sigma}
	Define a function $\Sigma: \Upsilon\times\Upsilon\to\frac{1}{4}\Z$ by
	\[
		\Sigma(g,h) = \frac{1}{4}\Tr \left(\frac{\phi(g) \bar\phi(h)}{\sqrt{-3}}\right).
	\]
	The function $\Sigma$ is an inhomogeneous $2$-cocycle on $\Upsilon$ with values in $\frac 14\Z$.
	Furthermore, the equation $\Sigma=\sigma$ holds in $H^2(\Upsilon,\frac{1}{12}\Z)$, where $\sigma$ is the cocycle defined
	in \autoref{eqn:sigma definition}.
\end{lemma}

\begin{proof}
	The 2-cocycle relation for $\Sigma$ follows immediately
	from the fact that $\phi$ and $\bar\phi$ are homomorphisms.
	To show that $\sigma=\Sigma$,
	we form the central extension corresponding to the cocycle $\sigma-\Sigma$, regarded as having values in $\frac{1}{12}\Z$:
	\begin{equation}
		\label{eqn:test extension}
		1 \to \frac{1}{12}\Z  \to \tilde\Upsilon \to \Upsilon \to 1.
	\end{equation}
	We must prove that this extension splits.
	
	We shall realize the group $\tilde{\Upsilon}$
	as the set $\Upsilon\times\frac{1}{12}\Z$, with the group
	operation defined by
	\[
		(g,x) (g',x')
		=(gg', x+x' + (\sigma-\Sigma)(g,g') ).
	\]
	Our first aim is to find a presentation for $\tilde{\Upsilon}$.
	For each of our generators $n_i$ above, we define a lift
	$\hat n_i = (n_i,0)$.
	The group $\tilde\Upsilon$ is generated by the lifts $\hat n_i$ together with one other element $z=(I_3,\frac{1}{12})$.
	
	For each of the thirteen relations $g_1^{a_1} \cdots g_t^{a_1} = I_3$ in
	$\Upsilon$,
	the corresponding element $\hat g_1^{a_1} \cdots \hat g_t^{a_t}\in\tilde{\Upsilon}$
	is in the central subgroup $\frac{1}{12}\Z$, and may be calculated easily using a computer.
	This gives us a corresponding relation in $\tilde\Upsilon$
	of the form
	\[
		\hat g_1^{a_1} \cdots \hat g_t^{a_t} = z^c.
	\]
	Of course, $z$ is in the centre of $\tilde\Upsilon$, so we
	have an additional relation $[z,\hat n_i]=(I_3,0)$ for each generator $n_i$.
	This gives us a presentation of $\tilde\Upsilon$ with six generators and eighteen relations.
	Just as before, we may use our presentation to calculate the
	abelianization $\tilde\Upsilon^{\ab}$.
	The result is that $\tilde\Upsilon^{\ab}$ is generated by
	$\hat{n}_1,\hat n_2,\hat n_3, \hat n_4, \hat n_5, z$ and has relations:
	\[
		\hat n_1^3 \hat n_4^3 =z^6,\quad
		\hat n_3^3 =z^6,\quad
		\hat n_5^3 = z^{-30}.
	\]
	In particular, there is a homomorphism $\Phi:\tilde\Upsilon\to\tilde{\Upsilon}^{\ab} \to \frac{1}{12}\Z$,
	defined by
	\[
		\Phi( \hat n_1) = \Phi(\hat n_4) = \Phi(z) = \frac{1}{12},
		\quad
		\Phi(\hat n_2) = 0,\quad
		\Phi(\hat n_3) = \frac{2}{12},\quad
		\Phi(\hat n_5) = \frac{-10}{12}. 
	\]
	The homomorphism $\Phi$ splits
	the extension in \autoref{eqn:test extension}.
\end{proof}

\begin{proof}[Proof of \autoref{thm:example}]
	Let $I$ be a non-zero ideal of $\Z[\zeta]$.
	By \autoref{prop:multiplier system cohomology},
	we must show that the image of $\sigma$ in $H^2(\Upsilon_{\nc}(2\cdot I),\Q/ (\Norm(I)\cdot \Z))$ is a coboundary.
	In view of \autoref{lemma:Sigma-sigma}, it is sufficient to prove this for the cocycle $\Sigma$.
	Let $g,h \in \Upsilon_{\nc}(2\cdot I)$.
	This means that $\phi(g)$ and $\phi(h)$ are in $2\cdot I$.
	Therefore $\phi(g)\bar{\phi}(h)$ is a multiple of $4\Norm(I)$.
	Using this, one easily checks that $\Sigma(g,h)$ is an
	integer, and is a multiple of $\Norm (I)$.
	Therefore $\Sigma(g,h)$ is zero in $\Q/(\Norm(I)\cdot \Z)$.
\end{proof}

\paragraph{The Multiplier System.}
We now describe the multiplier systems predicted in \autoref{thm:example}.
By \autoref{lemma:Sigma-sigma}, there is a function $\kappa : \Upsilon  \to \frac{1}{12}\Z$ such that
\begin{equation}
	\label{eqn:sigma Sigma split}
	\sigma(g,h) = \Sigma(g,h) + \kappa(gh) -\kappa(g) - \kappa(h),
	\qquad
	g,h\in\Upsilon.
\end{equation}
Of course, $\kappa$ is only defined up to addition of an element of $H^1(\Upsilon ,\frac{1}{12}\Z)$. This gives us the freedom to
insist (for example) that $\kappa$ is zero on the upper triangular generators $n_1$, $n_2$.
Once such a restriction is made, one can calculate the
function $\kappa$ using \autoref{eqn:sigma Sigma split}.

Choose any $g,h \in \Upsilon_{\nc}(2I)$. The number $\Sigma(g,h)$
is an integer and is a multiple of $\Norm(I)$.
This implies
\[
	\exp\left(2	\pi i \frac{\sigma(g,h)}{\Norm(I)}\right) =
	\exp\left(2	\pi i\frac{\kappa(gh)-\kappa(g)-\kappa(h)}{\Norm(I)}\right).
\]
Substituting the definition of $\sigma$ (\autoref{eqn:sigma definition}) we have for any $\tau$ in the symmetric space $\cH$:
\[
	\exp\left(\frac{\tilde j(gh,\tau) - \tilde j(g,h\tau) - \tilde{j}(h,\tau)}{\Norm(I)}\right) =
	\exp\left(2	\pi i\frac{\kappa(gh)-\kappa(g)-\kappa(h)}{\Norm(I)}\right).
\]
Here $\tilde j(g,\tau)$ is the logarithm branch of the multiplier system $j(g,\tau)$ defined in \autoref{eqn:j tilde}.
It follows that the
function $\ell(g,\tau)$ is a multiplier system
on $\Upsilon_{\nc}(2\cdot\Norm(I))$ of weight $\frac{1}{\Norm(I)}$,
where $\ell(g,\tau)$ is defined by
\begin{equation}
	\label{eqn:example multiplier system}
	\ell(g,\tau) = \exp\left(\frac{ \tilde j(g,\tau) - 2\pi i\cdot  \kappa(g)}{\Norm(I)}\right).
\end{equation}
Note that we have $\ell(g,\tau)^{12\cdot\Norm(I)} = j(g,\tau)^{12}$.

\begin{remark}
By working a little harder, one can improve on
\autoref{thm:example}, and show that $\Upsilon_{\nc}(I)$ has a multiplier system of weight $\frac{1}{\Norm (I)}$ if $\Norm (I)$ is odd, and
$\frac{2}{\Norm (I)}$ if $\Norm (I)$ is even.
\end{remark}

\section{Proof of \autoref{thm:main}}

\label{section:thm proof}

In this section, we shall assume \autoref{key lemma}, and deduce \autoref{thm:main}.
As before, we let $\Gamma$ be an arithmetic subgroup of $\SUd$ of the first kind.
Let $\sigma\in H^2(\Gamma,\Z)$ be the cocycle described in \autoref{section:background}.
We shall prove that for every positive integer $n$ there is
a subgroup $\Gamma_n$ of finite index in $\Gamma$, such that
the image of the cocycle $\sigma$ in $H^2(\Gamma_n,\Z/n)$ is
a coboundary.

\begin{proof}[Proof of \autoref{thm:main}]
	According to \autoref{key lemma} there exists a congruence subgroup $\Gamma'$ and elements $\phi_i,\psi_i \in H^1(\Gamma',\C)$
	such that
	\begin{equation}
		\label{eqn:cup product}
		\sigma = \sum_{i=1}^r \phi_i \cup \psi_i,
		\qquad
		\text{in }H^2(\Gamma',\C).
	\end{equation}
	Since $\sigma \in H^2(\Gamma',\Q)$, we may choose the
	elements $\phi_i,\psi_i$ to be in $H^1(\Gamma',\Q)$.
	We shall think of $\phi_i,\psi_i$ as homomorphisms $\Gamma' \to \Q$.
	The images of these homomorphisms are finitely generated, since arithmetic groups are finitely generated.
	We may therefore define a subgroup $\Gamma''$ of finite index in $\Gamma'$ by
	\[
		\Gamma'' = \bigcap_i \left(\phi_i^{-1}(\Z) \cap \psi_i^{-1}(\Z)\right).
	\]
	From now on we shall regard $\phi_i,\psi_i$ as elements
	of $H^1(\Gamma'',\Z)$.
	Note that \autoref{eqn:cup product} remains true in $H^2(\Gamma'',\Q)$ but not necessarily in $H^2(\Gamma'',\Z)$.
	We have a long exact sequence
	\[
		\to
		H^1(\Gamma'',\Q/\Z)
		\stackrel{\partial}{\to} H^2(\Gamma'',\Z)
		\to H^2(\Gamma'',\Q) \to . 
	\]
	From \autoref{eqn:cup product}, we have in $H^2(\Gamma'',\Z)$:
	\[
		\sigma
		= \sum_{i=1}^r ( \phi_i \cup \psi_i ) + \partial\gamma,
		\qquad
		\text{for some }\gamma \in H^1(\Gamma'',\Q/\Z).
	\]
	The subgroup $\Gamma_1 = \ker (\gamma : \Gamma''\to\Q/\Z)$ has finite index
	in $\Gamma''$, and in $H^2(\Gamma_1,\Z)$ we have
	\begin{equation}
		\label{eqn:cup product 3}
		\sigma
		= \sum_{i=1}^r \phi_i \cup \psi_i.
	\end{equation}
	Our tower of subgroups $\Gamma_n$ is defined by
	\begin{equation}
		\label{eq:subgp tower defn}
		\Gamma_n
		= \{ g \in \Gamma_1 : 	\phi_i(g) \equiv 0 \bmod n
		\text{ for $i=1,\ldots,r$} \}	
	\end{equation}
	The image of each $\phi_i$ in $H^1(\Gamma_n,\Z/n)$ is zero.
	Hence by \autoref{eqn:cup product 3} the image of $\sigma$ in $H^2(\Gamma_n,\Z/n)$ is zero.
\end{proof}

Before moving on, we investigate the index of the tower of subgroups $\Gamma_n$.
It is clear from the definition \autoref{eq:subgp tower defn} that we have $[\Gamma_1 : \Gamma_n] \le n^r$,
and this implies
\begin{equation}
	\label{eqn:index bound}
	[\Gamma : \Gamma_n] \ll n^r.
\end{equation}
On the other hand we have $\sigma=n\cdot \tau$ for the element
$\tau \in H^2(\Gamma_n,\Z)$ given by
$\tau = \sum \left(\frac{1}{n} \phi_i \right) \cup \psi_i$.
Hence $\sigma^d$ is a multiple of $n^d$
in $H^{2d}(\Gamma_n,\Z)$.

By Poincar\'e duality, we have a commutative diagram
\[
	\begin{matrix}
		H^{2d}(\Gamma,\Z)  & \cong & H_0(\Gamma,\Z) & \cong & \Z\\[3mm]
		\text{restriction}\downarrow\qquad & & \downarrow && \qquad\qquad\downarrow \times [\Gamma:\Gamma_n] \\[3mm]
		H^{2d}(\Gamma_n,\Z)  & \cong &H_0(\Gamma_n,\Z) & \cong & \Z
	\end{matrix}
\]
Since the restriction of $\sigma^d$ to $\Gamma_n$ is a multiple of $n^d$, we must have (as $n \to \infty$)
\begin{equation}
	\label{eqn:volume bound}
	n^d \ll [\Gamma: \Gamma_n].
\end{equation}
Comparing the two bounds in \autoref{eqn:index bound} and \autoref{eqn:volume bound}, we deduce the following:

\begin{corollary}
	\label{cor:rank bound}
	Let $\Gamma$ be a cocompact arithmetic subgroup of $\SUd$ of the first kind,
	and let $\sigma\in H^2(\Gamma,\Q)$ be the cohomology class of the invariant K\"ahler form.
	If $\sigma = \sum_{i=1}^r \phi_i \cup \psi_i$
	with $\phi_i,\psi_i\in H^1(\Gamma,\Q)$.
	Then $r \ge d$.
\end{corollary}

\section{Proof of \autoref{key lemma}}

\label{section:lemma proof}

Let $\Gamma$ be a cocompact arithmetic subgroup of $\SUd$ of the first kind.
Let $\sigma \in H^2_{\stable}(\C)$ denote the invariant K\"ahler class.
Our aim is to prove that $\sigma$ is in the subspace $V$
of $H^2_{\stable}$ spanned by cup products of
$H^1_{\stable}(\C) \otimes H^1_{\stable}(\C)$.

It is proved in \cite{BorelWallach} that $H^1(\Gamma',\C) \ne 0$ for some congruence subgroup $\Gamma'$ of $\Gamma$.
We shall choose such a subgroup $\Gamma'$ now.
Furthermore, by reducing the size of $\Gamma'$ if necessary we may assume that $\Gamma'$ is neat.
The neatness condition implies that the quotient space
$X = \Gamma'\backslash \cH$ is a smooth compact $d$-dimensional projective complex variety.
The cohomology of $\Gamma'$ is isomorphic to that of $X$.
We may therefore make use of Poincare duality and the Hard Lefschetz Theorem (see for example \cite{GriffithsHarris}).
In particular,
\begin{enumerate}
\item
$H^{2d}(\Gamma',\C)$ is a 1-dimensional vector space spanned by $\sigma^d$.
\item
The cup product induces a perfect pairing $H^1(\Gamma',\C) \otimes H^{2d-1}(\Gamma',\C) \to H^{2d}(\Gamma',\C)$.
\item
Cup product with $\sigma^{d-1}$ induces an isomorphism $H^1(\Gamma',\C) \to H^{2d-1}(\Gamma',\C)$.
\end{enumerate}

Choose a non-zero element $\phi\in H^1(\Gamma',\C)$.
By Poincar\'e duality, there exists $\phi^*\in H^{2d-1}(\Gamma',\C)$, such that $\phi \cup \phi^* = \sigma^d$.
By the Hard Lefschetz Theorem, $\phi^*=\psi \cup \sigma^{d-1}$ for some $\psi \in H^1(\Gamma',\C)$.
Thus the map
\[
	\Phi :V \to H^{2d}_{\stable}(\C),\qquad
	\Sigma \mapsto \Sigma \cup \sigma^{d-1}
\]
is surjective.

The element $\sigma$ is $\G(\A_f)$-invariant, so $\Phi$ is a morphism of $G(\A_f)$-modules.
The representation $V$ of $G(\A_f)$ is
contained in $H^2_{\stable}(\C)$ and is therefore semi-simple.
Therefore there exists a trivial $1$-dimensional subrepresentation in $V$.
This shows that $V^{G(\A_f)}\ne 0$.
On the other hand, $(H^2_{\stable}(\C))^{G(\A_f)}$ is spanned by $\sigma$, so we must have $\sigma \in V$.

\section{Explicit construction in the case $h^{2}=1$.}

In this section, we prove \autoref{thm:explicit},
which is a more precise version of
\autoref{key lemma} in a special case.
Our method leads to a bound on the cohomology
of certain arithmetic groups (\autoref{thm:H1 bound}).
We shall work throughout this section with rational cohomology groups $H^\bullet(\Gamma,\Q)$ rather than complex cohomology groups. The reason for this change is because we will need to
use \autoref{cor:rank bound} in order to prove \autoref{thm:H1 bound}.

As before we assume that $\Gamma$ is a cocompact arithmetic
subgroup of $\SUd$ of the first kind, but now we
make the additional assumption that $H^{2}(\Gamma,\C)$ is $1$-dimensional.
Since the cocycle $\sigma$ is a non-zero element of
$H^2(\Gamma,\Q)$, we have.
\[
	H^{2}(\Gamma,\Q) = \Span\{\sigma\}.
\]

By results in \cite{BorelWallach} we may choose a normal subgroup $\Gamma' \unlhd \Gamma$
of finite index, such that
\begin{enumerate}
	\item
	$\Gamma'$ is neat,
	\item
	$H^1(\Gamma',\Q) \ne 0$.
\end{enumerate}
We shall express $\sigma$ as a specific sum of cup products of elements in $H^1(\Gamma',\Q)$.
As a consequence, we prove in \autoref{thm:H1 bound} the bound
$\dim H^{1}(\Gamma',\C) \ge 2d$.
\medskip

Let $\cH$ be the symmetric space corresponding to $\SUd$.
The quotient $X = \Gamma' \backslash \cH$ is a compact complex manifold of dimension $d$ over the complex numbers, and the cohomology of $\Gamma'$ is canonically isomorphic to that of $X$.

Let $r$ be the dimension of $H^1(\Gamma',\Q)$ and choose a basis
$\{\phi_1, \ldots,\phi_r\}$ for $H^{1}(\Gamma',\Q)$.
By Poincar\'e duality, we many choose a dual basis $\phi_1^*,\ldots,\phi_r^*$ of $H^{2d-1}(\Gamma',\Q)$, i.e.
\begin{equation}
	\label{eqn:dual basis}
	\phi_i \cup \phi_j^* = \delta_{ij}\cdot \sigma^d.
\end{equation}
By the Hard Lefschetz Theorem, the cup product with $\sigma^{d-1}$ gives an isomorphism $H^{1}(\Gamma',\Q) \to H^{2d}(\Gamma',\Q)$.
We may therefore define $\psi_1, \ldots,\psi_r \in H^{1}(\Gamma',\Q)$ by
\[
	\psi_i \cup \sigma^{d-1} = \phi_i^*.
\]
Finally, we define an element $\Sigma\in H^{2}(\Gamma',\Q)$ by
\begin{equation}
	\label{eqn:Sigma defn}
	\Sigma
	= \frac{1}{r}\sum_{i=1}^r \left(\phi_i \cup \psi_i\right).
\end{equation}
Note that we obviously have $\Sigma \cup \sigma^{d-1} = \sigma^d$,
so $\Sigma$ is certainly non-zero.

\begin{theorem}
	\label{thm:explicit}
	Let $\Gamma$ be a cocompact subgroup of $\SUd$ of the first kind, such that
	$H^{2}(\Gamma,\Q)$ is $1$-dimensional.
	Let $\Gamma'$ be a neat normal subgroup satisfying
	$H^1(\Gamma',\Q) \ne 0$, and let $\Sigma\in H^{2}(\Gamma',\Q)$ be as defined in
	\autoref{eqn:Sigma defn}.
	Then we have $\Sigma = \sigma$ in $H^2(\Gamma',\Q)$.
\end{theorem}

We prove the theorem in two lemmata.

\begin{lemma}
	\label{lem:basis independence}
	The element $\Sigma\in H^{2}(X,\Q)$ does not depend on the choice of basis $\{\phi_1,\ldots,\phi_r\}$ of $H^{1}(X,\Q)$.
\end{lemma}

\begin{proof}
	For any finite dimensional vector space $V$, we may identify $\Hom(V,V)$ with
	$V \otimes V^*$. Under this identification, the identity element of $\Hom(V,V)$
	becomes $\sum \phi_i \otimes \phi_i^*$ for any basis $\phi_i$.
	Hence this sum is independent of the basis.
	The element $\Sigma$ is the image of such a sum under the following
	composition of maps:
	\[
		H^{1} \otimes (H^{1})^*
		\stackrel{\text{Poincar\'e duality}}{\longrightarrow} H^1 \otimes H^{2d-1}
		\stackrel{\text{Hard Lefschetz}}{\longrightarrow} H^{1} \otimes H^{1}
		\stackrel{\text{cup product}} \longrightarrow H^{2}.
	\]
\end{proof}

We are assuming in this section that $\Gamma'$ is normal in
$\Gamma$.
The finite group $\Gamma/\Gamma'$ acts on the cohomology
of $\Gamma'$.
This is a tiny version of the action of $\G(\A_f)$ on the stable cohomology.
The action is compatible with cup products, i.e. $g \phi \cup g\psi = g(\phi \cap \psi)$.
Furthermore $g\sigma=\sigma$ for all $g\in \Gamma/\Gamma'$.
The next lemma shows that the same is true of $\Sigma$.

\begin{lemma}
	For all $g\in \Gamma/\Gamma'$ we have $g(\Sigma) = \Sigma$.
\end{lemma}

\begin{proof}
	Let $v_i = g \phi_i$ for $i = 1,\ldots,r$.
	Clearly $\{v_i\}$ is another basis of $H^{1}(\Gamma',\Q)$,
	and we shall write $\{v_i^*\}$ for the dual basis in $H^{2d-1}(\Gamma',\Q)$.
	Applying $g$ to both sides of \autoref{eqn:dual basis}, we get (since $g\sigma=\sigma$):
	\[
		v_i \cup g\phi_j^* = \delta_{i,j}\cdot \sigma^d.
	\]
	Hence $v_i^* = g\phi_i^*$ for $i=1,\ldots,r$.
	By the hard Lefschetz theorem, there are unique elements $w_i\in H^1(\Gamma',\Q)$, such that $w_i \cup \sigma^{d-1} = v_i^*$.
	Since $g\sigma=\sigma$, we must have $w_i = g\psi_i$.
	Putting all of this together, we find
	\[
		g \Sigma
		=\frac{1}{r} \sum (g \phi_i) \cup (g\psi_i)\\
		=\frac{1}{r} \sum v_i \cup w_i.
	\]
	\autoref{lem:basis independence} implies that the right hand side is equal to $\Sigma$.
\end{proof}

\begin{proof}[Proof of \autoref{thm:explicit}]
	Recall that $H^{2}(\Gamma,\Q)$ is generated by $\sigma$.
	The previous lemma shows that $\Sigma \in H^2(\Gamma,\Q)^{\Gamma/\Gamma'}$.
	By the Hochschild--Serre spectral sequence, it follows that $\Sigma$ is the restriction of an element of $H^{2}(\Gamma,\Q)$.
	Therefore $\Sigma = c \cdot \sigma$ for some rational number $c$.
	The equation $\Sigma \cup \sigma^{d-1} = \sigma^d$ implies $c=1$.
\end{proof}

\begin{theorem}
	\label{thm:H1 bound}
	Let $\Gamma$ be a cocompact subgroup of $\SUd$ of the first kind, such that
	$H^{1,1}(\Gamma,\C)$ is $1$-dimensional.
	Let $\Gamma'$ be a neat normal subgroup 	satisfying $H^1(\Gamma',\C) \ne 0$,	Then
	\[
		\dim H^{1}(\Gamma',\C)
		\ge 
		2d.
	\]
\end{theorem}

\begin{proof}
Recall that according to \autoref{cor:rank bound}, we cannot express
$\sigma$ as a sum of fewer than $d$ cup products.
It is tempting to apply this fact directly to the formula
in \autoref{thm:explicit}. However, this gives only the bound $\dim H^{1}(\Gamma',\C)\ge d$, since the number of terms in the sum $\Sigma$ is $\dim H^1(\Gamma,\C)$.
To obtain the stronger bound of the theorem, we shall simplify our formula for $\Sigma$ a little further and remove half of the terms.

By Poincar\'e duality and the hard Lefschetz theorem,
we have a non-degenerate alternating form
\[
	{\bigwedge}^2 H^1(\Gamma',\Q) \to H^{2d}(\Gamma',\Q),
	\qquad
	\phi \wedge\psi \mapsto \phi \cup \psi \cup \sigma^{d-1}.
\]
Therefore the dimension $r$ of $H^1(\Gamma',\Q)$ is even, and we let $r=2s$.
Furthermore we may choose our basis $\{\phi_1,\ldots,\phi_{2s}\}$ to be a symplectic basis, i.e. one which satisfies
\[
	\phi_i \cup \phi_j \cup \sigma^{d-1}
	=
	\begin{cases}
		 \sigma^d & \text{if $j=i+s$},\\
		 -\sigma^d & \text{if $j=i-s$},\\
		 0 & \text{otherwise.}
	\end{cases}
\]
With such a choice of basis $\{\phi_i\}$ we have
$\psi_i=\phi_{i+s}$ if $i\le s$ and $\psi_i= - \phi_{i-s}$ if $i > s$.
This allows us to simplify our formula for $\Sigma$, so that we have
\[
	\Sigma = \frac{1}{s}\sum_{i=1}^s \phi_i \cup \phi_{i+r}.
\]
The key point here is that we are now expressing $\Sigma$ as a sum
of $s$ cup products, rather than a sum of $r$ cup products, so we have only half the number of terms as before.
The theorem now follows from \autoref{cor:rank bound}.
\end{proof}


\begin{thebibliography}{10}

\bibitem{BassMilnorSerre}
H. Bass, J. Milnor, J.-P. Serre,
``Solution of the congruence subgroup problem for $\SL_n$ ($n\ge 3$) and $\Sp_{2n}$ ($n \ge 2$)'',
Inst. Hautes Études Sci. Publ. Math. (1967) 59–137.

\bibitem{BlasiusRogawski}
D. Blasius, J. Rogawski,
``Cohomology of congruence subgroups of $\SU(2,1)^p$ and Hodge cycles on some special complex hyperbolic surfaces.''
\emph{Regulators in analysis, geometry and number theory}, 1–15, Progr. Math., 171, Birkhäuser Boston, Boston, MA, 2000.
 

\bibitem{BorelWallach}
A. Borel, N. Wallach,
\emph{Continuous cohomology, discrete subgroups, and representations of reductive groups.}
Second edition. Mathematical Surveys and Monographs, 67.
American Mathematical Society, Providence, RI, 2000.

\bibitem{Boyer}
P. Boyer,
``Sur la torsion dans la cohomologie des variétés de Shimura de Kottwitz-Harris-Taylor.''
J. Inst. Math. Jussieu 18 (2019), no. 3, 499–517.

\bibitem{Deligne}
P. Deligne,
``Extensions centrales non résiduellement finies de groupes arithmétiques'',
C. R. Acad. Sci. Paris Sér. A-B 287 (1978), no. 4, A203–A208.


\bibitem{Deligne2}
P. Deligne,
``Extensions centrales de groupes algébriques simplement
connexes et cohomologie galoisienne''
Publications mathématiques de l’I.H.É.S., tome 84 (1996), p. 35-89.


\bibitem{FalbelParker}
E. Falbel, J. R.  Parker,
``The geometry of the Eisenstein-Picard modular group''
Duke Math. J. 131 (2006), no. 2, 249–289.


\bibitem{GriffithsHarris}
P. Griffiths, J. Harris,
\emph{Principles of algebraic geometry.}
Wiley Classics Library. John Wiley \& Sons, Inc., New York, 1994.



\bibitem{hill code}
R. M. Hill,
sage code available at https://github.com/rmhi/SU21Eisenstein.

\bibitem{holzapfel}
R.-P. Holzapfel,
``Zeta dimension formula for Picard modular cusp forms of neat natural congruence subgroups.''
Abh. Math. Sem. Univ. Hamburg 68 (1998), 169–192.

\bibitem{ibukiyama}
T. Ibukiyama,
``Modular forms of rational weights and modular varieties.''
Abh. Math. Sem. Univ. Hamburg 70 (2000), 315–339.



\bibitem{Lina Jalal}
L. Jalal,
``A calculation of a half-integral weight
multiplier system on SU(2,1)''
Doctoral Thesis, University College London\\
(https://discovery.ucl.ac.uk/id/eprint/766292/1/766292.pdf)




\bibitem{Moore}
C. C. Moore, 
``Extensions and low dimensional cohomology theory of locally compact groups. I, II''.
Trans. Amer. Math. Soc. 113 (1964), 40–63; ibid. 113 (1964), 64–86.

\bibitem{Petersson III}
H. Petersson,
``Zur analytischen Theorie der Grenzkreisgruppen.
Teil III'',
Math. Ann. 115 (1938), 518–572.

\bibitem{Petersson IV}
H. Petersson,
``Zur analytischen Theorie der Grenzkreisgruppen.
Teil IV: Anwendungen der Formen- und Divisorentheorie'',
Math. Ann. 115 (1938), 670–709.


\bibitem{stover toledo}
M. Stover, M., D. Toledo,
``Residually finite lattices in $\widetilde{\PU(2,1)}$ and
fundamental groups of smooth projective surfaces'',
preprint 2021 (https://arxiv.org/abs/2105.12772).


\bibitem{WangWilliams}
H. Wang, B. Williams,
``Borcherds products of half-integral weight'',
preprint 2020 (https://export.arxiv.org/abs/2007.00055).


\end{thebibliography}
\end{document}